\documentclass{amsart}

\usepackage{fancyhdr}
\usepackage{lastpage}
\usepackage{stmaryrd,yhmath}

\newcommand{\bdis}{\begin{displaymath}}
\newcommand{\edis}{\end{displaymath}}
\newcommand{\be}{\begin{equation}}
\newcommand{\ee}{\end{equation}}
\newcommand{\mbb}{\mathbb}
\newcommand{\mcal}{\mathcal}

\newcommand{\vp}{\varphi}
\newcommand{\tT}{\tilde{T}}
\newcommand{\tU}{\tilde{U}}
\newcommand{\btau}{\bar{\tau}}
\newcommand{\vth}{\vartheta}

\newtheorem{theorem}{Theorem}
\newtheorem{lemma}[]{Lemma}

\theoremstyle{definition}

\theoremstyle{remark}
\newtheorem{remark}[]{Remark}

\newtheorem*{mydef2}{{\bf Definition}}

\numberwithin{equation}{section}

\begin{document}

\title{Jacob's ladders and the asymptotically approximate solutions of a nonlinear diophantine equation}

\author{Jan Moser}

\address{Department of Mathematical Analysis and Numerical Mathematics, Comenius University, Mlynska Dolina M105, 842 48 Bratislava, SLOVAKIA}

\email{jan.mozer@fmph.uniba.sk}


\begin{abstract}
The nonlinear equation which is connected with the main term of the Hardy-Littlewood formula  for $\zeta^2(1/2+it)$ is studied. In this direction
I obtain the fine results which cannot be reached by published methods of Balasubramanian, Heath-Brown and Ivic in the field of the
Hardy-Littlewood integral.
\end{abstract}

\maketitle

\section{Formulation of the results}

\subsection{}

Let

\be \label{1.1}
S(T,U)=2\sum_{n<P}\frac{d(n)}{\sqrt{n}}\frac{\sin\left(\frac U2\ln\frac Pn\right)}{\frac{U}{2}\ln\frac{P}{n}}
\cos\left\{\left( 2\pi P+\frac U2\right)\ln\frac Pn-2\pi P-\frac \pi 4\right\} ,
\ee
where $P=T/2\pi$ and $d(n)$ is the number of divisors of $n$. In this paper I consider the nonlinear diophantine equation
\be \label{1.2}
\tau=\frac{S(T,U)}{\ln T},\ \tau\in [\eta,1-\eta]\bigcup \{ 1\}
\ee
in two variables $T,U$ with the parameter $\tau$ where
\be \label{1.3}
T\in [T_0,T_0+U_0],\ U\in \left( \left. 0,T_0^{1/6-\epsilon/2}\right.\right],\ U_0=T_0^{1/3+2\epsilon} ,
\ee
and $0<\eta$ is a sufficiently small fixed number  and $0<T_0$ is a sufficiently big fixed number.

\begin{mydef2}

Let for $\bar{\tau}\in [\eta,1-\eta]\cup\{ 1\}$ there be a sequence $\{ T_0(\bar{\tau})\},\ T_0\to\infty$ and the values
$\tT=\tT(T_0,\btau),\ \tU=\tU(T_0,\btau)$ for which
\begin{eqnarray} \label{1.4}
& &
\tT\in [T_0,T_0+1.1U_0],\ \tU\in\left(\left. 0,T_0^{1/6-\epsilon/2}\right.\right] , \\
& &
\btau\sim \frac{S(\tT,\tU)}{\ln\tT},\ T_0\to\infty \nonumber
\end{eqnarray}
is fulfilled. Then the pair $[\tT,\tU]$ is called the \emph{asymptotically approximate solution} (AAS) of the equation (1.2) for $\tau=\btau$.

\end{mydef2}

\subsection{}

The method of parallel and rotating chords (see \cite{4}-\cite{6}) leads to the proof of the following theorems.

\begin{theorem}
For $\tau=1$ there is the continuum AAS of the equation
\bdis
1=\frac{S(T,U)}{\ln T} .
\edis
The structure of the set of these solutions is such as follows: to each sufficiently big $T_0$ continuum of AAS corresponds.
\end{theorem}

\begin{theorem}
Let ${\gamma}$ denote the sequence of the zeroes of $\zeta(1/2+it)$. Then for each $\tau\in [\eta,1-\eta]$ there is a continuum of the AAS of the
equation
\bdis
\tau=\frac{S(T,U)}{\ln T} .
\edis
The structure of the set of these solutions is such as follows: to each sufficiently big $\gamma$ continuum of AAS corresponds.
\end{theorem}

\begin{remark}

It is clear that these results cannot be reached by published methods of Balasubramanian, Heath-Brown and Ivic in the field of the
Hardy-Littlewood integral.

\end{remark}

This paper is a continuation of the series of works \cite{4}-\cite{11}.

\section{Lemmas}

\subsection{}

\begin{lemma}
\begin{eqnarray}  \label{2.1}
& &
\int_T^{T+U}\cos\{ 2\vth(t)-t\ln n\}{\rm d}t= \\
& &
=U\frac{\sin\left(\frac U2\ln\frac Pn\right)}{\frac U2\ln\frac Pn}
\cos\left\{\left( 2\pi P+\frac U2\right)\ln \frac Pn-2\pi P-\frac \pi 4\right\}+ \nonumber \\
& &
+\mcal{O}\left(\frac{U+U^3}{T}\right),\
P=\frac{T}{2\pi} , \nonumber
\end{eqnarray}
for $U>0$, where
\bdis
\vth(t)=-\frac{1}{2}t\ln\pi +\text{Im}\left\{ \ln\Gamma\left(\frac 14+\frac 12it\right)\right\} .
\edis
\end{lemma}

\begin{proof}
Following the formulae (see \cite{7}, pp. 221, 329)
\begin{eqnarray} \label{2.2}
& &
\vth(t)=\frac{1}{2}t\ln\frac{t}{2\pi}-\frac{1}{2}t-\frac 18\pi+\mcal{O}\left(\frac{1}{t}\right) , \nonumber \\
& & \\
& &
\vth^\prime(t)=\frac{1}{2}\ln\frac{t}{2\pi}+\mcal{O}\left(\frac{1}{t}\right) \nonumber
\end{eqnarray}
we have ($t=T+x,\ x\in [0,U]$)
\begin{eqnarray}
& &
2\vth(T+x)-(T+x)\ln n= \\
& &
x\ln\frac Pn +2\pi P\ln \frac Pn-2\pi P-\frac \pi 4+\mcal{O}\left(\frac{1+U+U^3}{T}\right) . \nonumber
\end{eqnarray}
Then we obtain (2.1) by (2.3).
\end{proof}

\subsection{}

\begin{lemma}
\be \label{2.4}
\frac 1U\int_U^{T+U} Z^2(t){\rm d}t=S(T,U)+\mcal{O}\left(\frac{1}{T^{1/6}}\right) ,
\ee
for $U\leq T^{1/6-\epsilon/2}$.
\end{lemma}
\begin{proof}
Let us remind the Hardy-Littlewood formula (see \cite{12}, p. 80)
\be \label{2.5}
Z^2(t)=2\sum_{n\leq \frac{t}{2\pi}}\frac{d(n)}{\sqrt{n}}\cos\{ 2\vth(t)-t\ln n\}+\mcal{O}(t^{-1/6})
\ee
with the Motohashi error term (see \cite{3}, p. 125). Since
\bdis
\sum_{\frac{T}{2\pi}\leq n\leq \frac{T+U}{2\pi}}\frac{d(n)}{\sqrt{n}}=\mcal{O}\left( UT^\epsilon\frac{1}{\sqrt{T}}\right)=
\mcal{O}\left(\frac{U}{T^{1/2-\epsilon}}\right) ,
\edis
then
\begin{eqnarray*}
& &
Z^2(t)=2\sum_{n< P}\frac{d(n)}{\sqrt{n}}\cos\{ 2\vth(t)-t\ln n\}+ \\
& &
\mcal{O}\left(\frac{1}{T^{1/6}}\right)+\mcal{O}\left(\frac{U}{T^{1/2-\epsilon}}\right) ,
\end{eqnarray*}
and
\begin{eqnarray} \label{2.6}
& &
\int_T^{T+U} Z^2(t){\rm d}t=2\sum_{n< P}\frac{d(n)}{\sqrt{n}}\int_T^{T+U}\cos\{ 2\vth(t)-t\ln n\}{\rm d}t+ \\
& &
+\mcal{O}\left(\frac{U}{T^{1/6}}\right)+\mcal{O}\left(\frac{U^2}{T^{1/2-\epsilon}}\right) . \nonumber
\end{eqnarray}
Since
\bdis
\frac{U+U^3}{T}\sum_{n<P}\frac{d(n)}{\sqrt{n}}=\mcal{O}\left(\frac{U+U^3}{T}T^\epsilon\sqrt{T}\right)=
\mcal{O}\left( \frac{U+U^3}{T^{1/2-\epsilon}}\right),
\edis
then we obtain (2.4) by (2.1), (2.6).
\end{proof}

\section{Proofs of the Theorems}

\subsection{Proof of Theorem 1}

Let us remind that for every sufficiently big $T_0$ there is a continuum of pairs
\be \label{3.1}
[\tT,\tU]:\ \tT\in [T_0,T_0+U_0],\ \tU\in\left(\left. 0,T^{1/6-\epsilon/2}\right.\right]
\ee
for which the formula
\be \label{3.2}
\frac{1}{\tU}\int_{\tT}^{\tT+\tU}Z^2(t){\rm d}t=\ln\tT\left\{ 1+\mcal{O}\left(\frac{\ln\ln\tT}{\ln \tT}\right)\right\}
\ee
is true (see \cite{6}, Corollary 2, Remark 4). Then we obtain
\be \label{3.3}
1+\mcal{O}\left(\frac{\ln\ln\tT}{\ln \tT}\right)=\frac{S(\tT,\tU)}{\ln\tT}+\mcal{O}\left(\frac{1}{\tT^{1/6}\ln\tT}\right),\ T_0\to\infty
\ee
by (2.4), (3.2). Finally, we obtain the assertion by (3.1), (3.3).

\subsection{Proof of Theorem 2}

First of all the formula
\begin{eqnarray} \label{3.4}
& &
\int_\gamma^{\gamma+U(\gamma,\alpha)}Z^2(t){\rm d}t=\tau U\ln\gamma\left\{ 1+\mcal{O}\left(\frac{\ln\ln\gamma}{\ln\gamma}\right)\right\}, \\
& &
\tau=\tan[\alpha(\gamma,U)]\in [\eta,1-\eta],\ U(\gamma)=\gamma^{1/3+2\epsilon}+\Delta(\gamma)<1.1\gamma^{1/3+2\epsilon} \nonumber
\end{eqnarray}
for rotating chord is true ((3.4) is the dual asymptotic formula which corresponds to \cite{5}, (4.4) by \cite{6}, (1.2)). It is clear that for
every fixed direction of the rotating chord (i.e. for the fixed value $\tau\in [\eta,1-\eta]$) a continuum of parallel chords corresponds. From this set
we choose a continual subset such that the condition
\be \label{3.5}
\tU<\gamma^{1/6-\epsilon/2}
\ee
(compare with \cite{6}, Remark 4) is fulfilled. For this continuum set the formula
\be \label{3.6}
\frac{1}{\tU}\int_{\tT}^{\tT+\tU}Z^2(t){\rm d}t=\tau\ln\tT\left\{ 1+\mcal{O}\left(\frac{\ln\ln\tT}{\ln\tT}\right)\right\}
\ee
is true (see (3.4)), where
\be \label{3.7}
\gamma<\tT<\gamma+1.1U_0(\gamma);\ U_0(\gamma)=\gamma^{1/3+2\epsilon} .
\ee
Next, from (3.6) by (2.4) we obtain
\be\label{3.8}
\tau=\frac{S(\tT,\tU)}{\ln \tT}+\mcal{O}\left(\frac{\ln\ln\tT}{\ln\tT}\right),\ \gamma\to\infty .
\ee
Finally, we obtain, by (3.5), (3.7), (3.8) the assertion.

\section{Discussion on necessity of a new theory for short and microscopic parts of the Hardy-Littlewood integral}

Let us remind the results of Balasubramanian, Heath-Brown and Ivic for the Hardy-Littlewood integral
\bdis
\int_0^T Z^2(t){\rm d}t
\edis
and for the parts of this integral.

\subsection{}

First of all the Balasubramanian formula
\be \label{4.1}
\int_0^T Z^2(t){\rm d}t=T\ln T+(2c-1-\ln 2\pi)T+R(T),\ R(T)=\mcal{O}(T^{1/3+\epsilon})
\ee
is true (see \cite{1}).

\begin{remark}
The Good's $\Omega$-theorem (see \cite{2}) implies for the Balasubramanian's formula (4.1) that
\be \label{4.2}
\limsup_{T\to\infty}|R(T)|=+\infty ,
\ee
i.e. the error term in (4.1) is unbounded at $T\to\infty$.
\end{remark}

For the short interval the Balasubramanian's formula implies
\be \label{4.3}
\int_T^{T+U_0}Z^2(t){\rm d}t=U_0\ln T+(2c-\ln 2\pi)U_0+\mcal{O}(T^{1/3+\epsilon}),\ U_0=T^{1/3+2\epsilon} .
\ee

\subsection{}

Furthermore, let us remind the Heath-Brown's estimate (see \cite{3}, (7.20), p. 178)
\begin{eqnarray} \label{4.4}
& &
\int_{T-G}^{T+G}Z^2(t){\rm d}t=\mcal{O}\left\{ G\ln T+G\sum_K (TK)^{-\frac{1}{4}}\left( |S(K)|+\right.\right. \\
& &
\left.\left. +K^{-1}\int_0^K|S(x)|{\rm d}x\right)e^{-\frac{G^2K}{T}}\right\} \nonumber
\end{eqnarray}
(for definition of used symbols see \cite{3}, (7.21)-(7.23)), uniformly for $T^\epsilon\leq G\leq T{1/2-\epsilon}$. And, finally, we add the Ivic'
estimate (\cite{3}, (7.26))
\be \label{4.5}
\int_{T-G}^{T+G}Z^2(t){\rm d}t=\mcal{O}(G\ln^2 T),\ G\geq T^{1/3-\epsilon_0},\ \epsilon_0=\frac{1}{108}\approx 0.009 .
\ee

\begin{remark}
It is quite evident that the intervals $[T-G,T+G], G\in (0,1)$, for example, cannot be reached in theories leading to the formula (4.3) of
Balasubramanian or to the estimates (4.4) and (4.5) of Heath-Brown and Ivic, respectively.
\end{remark}

\subsection{}

In this situation I developed the new theory based on geometric properties of the Jacob's ladders. Let us remind the basic formulae of our theory. \\

Titchmarsh-Kober-Atkinson (TKA) formula (see \cite{12}, p. 141)
\be \label{4.6}
\int_0^\infty Z^2(t)e^{-2\delta t}{\rm d}t=\frac{c-\ln(4\pi\delta)}{2\sin\delta}+\sum_{n=0}^N c_n\delta^n+\mcal{O}(\delta^{N+1})
\ee
remained as an isolated result for the period of 56 years. We have discovered (see \cite{4}) the nonlinear integral equation
\be \label{4.7}
\int_0^{\mu[x(T)]}Z^2(t)e^{-\frac{2}{x(T)}t}{\rm d}t=\int_0^T Z^2(t){\rm d}t ,
\ee
in which the essence of the TKA formula is encoded. Namely, we have shown in \cite{4} that the following almost-exact formula for the Hardy-Littlewood
integral (after the period of 90 years)
\be \label{4.8}
\int_0^T Z^2(t){\rm d}t=\frac{\vp(T)}{2}\ln\frac{\vp(T)}{2}+(c-\ln 2\pi)\frac{\vp(T)}{2}+c_0+\mcal{O}\left(\frac{\ln T}{T}\right)
\ee
takes place, where $\vp(T)$ is the Jacob's ladder, i.e. an arbitrary solution to the nonlinear integral equation (4.7).

\begin{remark}
In the case of our result (4.8) the error term tends to zero as $T$ goes to infinity, namely
\bdis
\lim_{T\to\infty} r(T)=0,\ r(T)=\mcal{O}\left(\frac{\ln T}{T}\right) ,
\edis
(compare with (4.2)).
\end{remark}

\subsection{}

In the papers \cite{4},\cite{5} I obtained the following additive formula
\be \label{4.9}
\int_T^{T+U}Z^2(t){\rm d}t=U\ln\left( \frac{\vp(T)}{2}e^{-a}\right)\tan[\alpha(T,U)]+\mcal{O}\left(\frac{1}{T^{1/3-4\epsilon}}\right)
\ee
that holds true for short parts of the Hardy-Littlewood integral. Next, in the paper \cite{6} I proved the multiplicative asymptotic formula
($\mu[\vp]=7\vp\ln\vp$)
\be\label{4.10}
\int_T^{T+U}Z^2(t){\rm d}t=U\ln T\tan[\alpha(T,U)]\left\{ 1+\mcal{O}\left(\frac{\ln\ln T}{\ln T}\right)\right\},\
U\in\left(\left. 0,\frac{T}{\ln T}\right]\right.
\ee
for short and microscopic parts of the Hardy-Littlewood integral.

\begin{remark}
The formulae (4.7)-(4.10) - and all corollaries from these (see \cite{5}, \cite{6}) - cannot e derived within complicated methods of
Balasubramanian, Heath-Brown and Ivic. This proves the necessity of a new method - which is based on elementary geometric properties of
Jacob's ladders - to study short and microscopic parts  of the Hardy-Littlewood integral.
\end{remark}

\thanks{I would like to thank Michal Demetrian for helping me with the electronic version of this work.}


\begin{thebibliography}{19}
%
%
%
\bibitem{1}
R. Balasubramanian, `An improvement on a theorem of Titchmarsh on the mean quare of $|\zeta(1/2+it)|$`, Proc. Lond. Math. Soc. 3, 36 (1978) 540-575.
%
\bibitem{2}
A. Good, `Ein $\Omega$-Resultat f\" ur quadratische Mittel der Riemannschen Zetafunktion auf der kritische Linie',
Invent. Math. 41 (1977), 233-251.
%
\bibitem{3}
A. Ivic, `The Riemann zeta-function`, A Willey-Interscience Publication, New York, 1985.
%
\bibitem{4}
J. Moser, `Jacob's ladders and the almost exact asymptotic representation of the Hardy-Littlewood integral', (2008), arXiv:0901.3973.
%
\bibitem{5}
J. Moser, `Jacob's ladders and the tangent law for short parts of the Hardy-Littlewood integral', (2009), arXiv:0906.0659.
%
\bibitem{6}
J. Moser, `Jacob's ladders and the multiplicative asymptotic formula for short and microscopic parts of the Hardy-Littlewood integral', (2009),
arXiv:0907.0301.
%
\bibitem{7}
J. Moser, `Jacob's ladders and the quantization of the Hardy-Littlewood integral', (2009), arXiv:0909.3928.
%
\bibitem{8}
J. Moser,
`Jacob's ladders and the first asymptotic formula for the expression of the sixth order $|\zeta(1/2+i\varphi(t)/2)|^4|\zeta(1/2+it)|^2$', (2009),
arXiv:0911.1246.
%
\bibitem{9}
J. Moser,
`Jacob's ladders and the first asymptotic formula for the expression of the fifth order
$Z[\varphi(t)/2+\rho_1]Z[\varphi(t)/2+\rho_2]Z[\varphi(t)/2+\rho_3]\hat{Z}^2(t)$ for the collection of disconnected sets`, (2009),
arXiv:0912.0130.
%
\bibitem{10}
J. Moser, `Jacob's ladders, the iterations of Jacob's ladder $\vp_1^k(t)$ and asymptotic formulae for the integrals of the products
$Z^2[\varphi^n_1(t)]Z^2[\varphi^{n-1}(t)]\cdots Z^2[\varphi^0_1(t)]$ for arbitrary fixed $n\in \mbb{N}$` (2010), arXiv:1001.1632.
%
\bibitem{11}
J. Moser, `Jacob's ladders and the asymptotic formula for the integral of the eight order expression
$|\zeta(1/2+i\vp_2(t))|^4|\zeta(1/2+it)|^4$`, (2010), arXiv:1001.2114.
%
\bibitem{12}
E.C. Titchmarsh,
`The theory of the Riemann zeta-function`, Clarendon Press, Oxford, 1951.


\end{thebibliography}
\end{document}